\documentclass[11pt]{article}
\usepackage{mathrsfs}
\usepackage{amsfonts}
\usepackage{amsmath}
\usepackage{amssymb}
\usepackage{latexsym}
\usepackage{amsthm}
\usepackage{indentfirst}
\usepackage{graphicx}
\usepackage{cite}
\usepackage{bm}
\usepackage{algorithmic}
\usepackage[algosection,ruled]{algorithm2e}

\newtheorem{theorem}{Theorem}[section]
\newtheorem{lemma}{Lemma}[section]
\newtheorem{corollary}{Corollary}[section]

\theoremstyle{definition}

\newcommand{\I}{\bm{\mathrm{I}}}
\newcommand{\me}{\mathrm{e}}

\newcommand{\ball}{\bm{\mathrm{B}}}

\newcommand{\NS}{\mathbb{N}}
\newcommand{\RS}{\mathbb{R}}
\newcommand{\inv}[1]{{#1}^{-1}}
\newcommand{\tran}[1]{{#1}^\mathrm{\top}}
\newcommand{\diag}{\mathrm{diag}}

\renewcommand{\vec}[1]{\bm{\mathrm{#1}}}

\newcommand{\vc}{\bm{\mathrm{c}}}
\newcommand{\vp}{\bm{\mathrm{p}}}
\newcommand{\vq}{\bm{\mathrm{q}}}
\newcommand{\vrho}{\bm{\mathrm{\rho}}}
\newcommand{\vdc}{\Delta\bm{\mathrm{\vc}}}

\newcommand{\udots}{\mathinner{\mskip1mu\raise1pt\vbox{\kern7pt\hbox{.}}
\mskip2mu\raise4pt\hbox{.}\mskip2mu\raise7pt\hbox{.}\mskip1mu}}

\numberwithin{equation}{section}

\allowdisplaybreaks

\begin{document}

\title{A Globally Convergent Inexact Newton-Like Cayley Transform Method for
 Inverse Eigenvalue Problems}
\author{Yonghui Ling
\footnote{Corresponding author.
\newline \indent
    {\it E-mail:} lingyinghui@163.com (Y. Ling), xxu@zjnu.edu.cn (X. Xu).}
\\
Department of  Mathematics, Zhejiang University, Hangzhou 310027,
China
\\
Xiubin Xu \footnote{The second author's work was supported in part
by the National Natural Science Foundation of China (Grant No.
61170109).}
\\
Department of Mathematics, Zhejiang Normal University, Jinhua
321004, China }

\maketitle

\begin{abstract}
We propose a inexact Newton method for solving inverse eigenvalue
problems (IEP). This method is globalized by employing the classical
backtracking techniques. A global convergence analysis of this
method is provided and the R-order convergence property is proved
under some mild assumptions. Numerical examples demonstrate that the
proposed method is very effective for solving the IEP with distinct
eigenvalues.
\end{abstract}

\section{Introduction}

In the present paper, we consider inverse eigenvalue problems (IEP)
which are defined as follows. Let $\vec{c} = \tran{(c_1, c_2,
\ldots, c_n)} \in \RS^n$ and $\{A_i\}_{i = 0}^n$ be a sequence of
real symmetric $n \times n$ matrices. Define
\begin{equation}
\label{mat:A(c)} A(\vec{c}) = A_0 + \sum_{i = 1}^n c_i A_i
\end{equation}
and denote its eigenvalues by $\{\lambda_i(\vec{c})\}_{i = 1}^n$
with the increasing order $\lambda_1(\vec{c}) \leq
\lambda_2(\vec{c}) \leq \cdots \leq \lambda_n(\vec{c})$. Given $n$
real numbers $\{\lambda_i^*\}_{i = 1}^n$ which are arranged in
increasing order $\lambda_1^* \leq \lambda_2^* \leq \cdots \leq
\lambda_n^*$, the IEP is to find a vector $\vec{c}^* \in \RS^n$ such
that
\begin{equation}
\label{eq:IEP} \lambda_i(\vec{c}^*) = \lambda_i^*, \ \ i = 1, 2,
\ldots, n.
\end{equation}
Such vector $\vec{c}^*$  is called a solution of the IEP. This type
of inverse problem arises in a variety of applications, for
instance, the inverse Toeplitz eigenvalue problem
\cite{Trench1997,Diele2004}, inverse Sturm-Liouville's problem,
inverse vibrating string problem and the pole assignment problem,
see \cite{Chu1998,ChuGolub2002,ChuGolub2005} and the references
therein for more details on these applications.

Define $f: \RS^n \to \RS^n$ by
\begin{equation}
\label{fun:f(c)} f(\vec{c}) = \tran{(\lambda_1(\vec{c}) -
\lambda_1^*, \lambda_2(\vec{c}) - \lambda_2^*, \ldots,
\lambda_n(\vec{c}) - \lambda_n^*)}.
\end{equation}
Then, solving IEP (\ref{eq:IEP}) is equivalent to solving the
nonlinear equation $f(\vec{c}) = \vec{0}$ on $\RS^n$. It is clear
that, $\vec{c}^*$ is a solution of the IEP if and only if
$\vec{c}^*$ is a solution of the equation $f(\vec{c}) = \vec{0}$.
Based on this equivalence, Newton's method can be applied to the
IEP, and it converges quadratically \cite{Friedland1987}. As it is
known, each iteration of Newton's method involves solving a complete
eigenproblem for the matrix $A(\vec{c})$. To overcome this drawback,
different Newton-like methods have been proposed and studied
\cite{ChanXuZhou1999, ChanChungXu2003}.  To alleviate the
over-solving problem, Bai, Chan and Morini presented in
\cite{BaiChanMorini2004} an inexact Cayley transform method for
solving the nonlinear system $f(\vec{c}) = \vec{0}$. To avoid the
computation of the approximate Jacobian equations, Shen and Li
proposed in \cite{ShenLiJin2011, ShenLi2012} Ulm-like methods for
solving the IEPs. However, all these numerical methods for solving
the IEP converge only locally.

In this paper, we study the numerical methods with global
convergence property for solving the IEP.  Since the IEP is a
nonlinear equation, we review some classical work on solving the
general nonlinear equation $f(\vc) = \vec{0}$. Among the inexact
Newton-type methods where a line-search procedure is used, an
inexact Newton backtracking method was proposed in
\cite{Eisenstat1994}. It performed backtracking along the inexact
Newton step, and computational results on a large set of test
problems have shown its robustness and efficiency
\cite{Eisenstat1996,Pernice1998}.

The purpose of the present paper is, motivated by the inexact Newton
backtracking method proposed in \cite{Eisenstat1994}, to propose an
inexact Newton-type method which combines with the Cayley transform
method for solving the IEP. Under the classical assumption, which is
also used in \cite{BaiChanMorini2004,ChanChungXu2003,ShenLi2012},
that the given eigenvalues are distinct and the Jacobian matrix
$J(\vc^*)$ is invertible, we show that this method is global
convergent. To further improve the feasibility, we also propose a
hybrid method for solving the IEP by using a simpler condition in
the backtracking loop. Some numerical examples are reported to
illustrate the effectiveness of the proposed hybrid method with
distinct eigenvalues.

The paper is organized as follows. In Section 2, a global inexact
Newton-type algorithm is proposed. The global convergence analysis
is given in Section 3. In Section 4, we present a hybrid algorithm
for solving the IEP. And finally in Section 5, some numerical
examples are given to confirm the numerical effectiveness and the
good performance of our algorithm.

\section{A Global Inexact Newton-Type Method}

In this section, we present our algorithm. Let $\RS^{n \times n}$
denote the set of all real $n \times n$ matrices. Let $\|\cdot\|$
and $\|\cdot\|_F$ denote the 2-norm and the Frobenius norm in
$\RS^n$, respectively. The induced 2-norm in $\RS^{n \times n}$ is
also denoted by $\|\cdot\|$, i.e.,
$$
\|A\| := \sup_{\vec{x} \in \RS^n, \vec{x} \neq \vec{0}} \frac{\|A
\vec{x}\|}{\|\vec{x}\|}, \ \ \ A \in \RS^{n \times n}.
$$
Then, we have $\|A\| \leq \|A\|_F$ for any $A \in \RS^{n \times n}$.
Let $\lambda_1(\vec{c}) \leq \lambda_2 (\vec{c}) \leq \cdots \leq
\lambda_n (\vec{c})$ be the eigenvalues of matrix $A(\vec{c})$, and
let $\{\vec{q}_i(\vec{c})\}_{i = 1}^n$ be the normalized
eigenvectors corresponding to $\{\lambda_i(\vec{c})\}_{i = 1}^n$.
Define $J(\vec{c}) = ([J(\vec{c})]_{ij})$ by
\begin{equation}
\label{comp:J(c)ij} [J(\vec{c})]_{ij} := \tran{\vec{q}_i(\vec{c})}
A_j \vec{q}_i(\vec{c}), \ \ \ 1 \leq i, j \leq n.
\end{equation}
Let $\{\lambda_i^*\}_{i = 1}^n$ be given with $\lambda_1^* \leq
\lambda_2^* \leq \cdots \leq \lambda_n^*$ and write $\vec{\lambda}^*
= \tran{(\lambda_1^*, \lambda_2^*, \ldots, \lambda_n^*)}$.

The Cayley transform method for computing approximately the
eigenproblem of the matrix $A(\vc)$ was proposed in
\cite{Friedland1987} and was applied in
\cite{BaiChanMorini2004,ShenLi2012}. We now recall this method and
then apply it to our algorithm. Suppose that $\vc^*$ is a solution
of the IEP. Then, there exists an orthogonal matrix $Q_*$ such that
\begin{equation}
\label{eq:Q*A(c*)Q*} \tran{Q}_* A(\vc^*) Q_* = \diag\{\lambda_1^*,
\lambda_2^*, \ldots, \lambda_n^*\} \triangleq \Lambda^*.
\end{equation}
Assume that $\vc, \vrho$ and $P$ are the current approximations of
$\vc^*, \vec{\lambda}^*$ and $Q_*$, respectively. Define $\me^Y :=
\tran{P}Q_*$, where $Y$ is a skew-symmetric matrix. Then,
(\ref{eq:Q*A(c*)Q*}) can be rewritten as:
\begin{align}
\tran{P} A(\vc^*) P & = \me^Y \Lambda^* \me^{-Y} \nonumber\\
& = (\I + Y + \frac{1}{2}Y^2 + \cdots) \Lambda^* (\I - Y +
\frac{1}{2}Y^2 + \cdots) \nonumber\\
& = \Lambda^* + Y \Lambda^* - \Lambda^* Y + O(\|Y\|^2).
\label{eq:PA(c*)P}
\end{align}
Based on (\ref{eq:PA(c*)P}), we define the new approximation
$\vc^{\text{new}}$ of $\vc^*$ by neglecting the second-order terms
in $Y$:
\begin{equation}
\label{eq:PA(cnew)P} \tran{P} A(\vc^{\text{new}}) P = \Lambda^* + Y
\Lambda^* - \Lambda^* Y.
\end{equation}
By equating the diagonal elements in (\ref{eq:PA(cnew)P}), we have
$$
\lambda_i^* = \tran{\vec{p}_i(\vc)} A(\vc^{\text{new}})
\vec{p}_i(\vc), \ \ \ i = 1, 2, \ldots, n,
$$
where $\{\vec{p}_i(\vc)\}_{i = 1}^n$ are the column vectors of $P$.
Thus, once we get $\vc^{\text{new}}$ by solving the Jacobian
equation, we can obtain $Y$ by equating the off-diagonal elements in
(\ref{eq:PA(cnew)P}), i.e.,
\begin{equation}
\label{comp:Yij} [Y]_{ij} := \frac{\tran{\vec{p}_i(\vc)}
A(\vc^{\text{new}}) \vec{p}_j(\vc)}{\lambda_j^* - \lambda_i^*}, \ \
\ 1 \leq i \neq j \leq n.
\end{equation}

In order to update the new approximation  $P^{\text{new}}$ of $Q_*$,
we construct an orthogonal matrix $U$ using Cayley's transform
\begin{equation}
\label{mat:U} U := (\I + \frac{1}{2} Y)\inv{(\I - \frac{1}{2}Y)}
\end{equation}
and set $P^{\text{new}} = P U$, that is, we can obtain
$P^{\text{new}}$ by solving
\begin{equation}
\label{eq:solve_Pnew} (\I + \frac{1}{2} Y) P^{\text{new}} = (\I -
\frac{1}{2} Y) P.
\end{equation}
Finally, the new approximations of eigenvalues can be obtained by
\begin{equation}
\label{comp:rhoinew} \rho_i^{\text{new}} :=
\tran{\vec{p}_i(\vc^{\text{new}})} A(\vc^{\text{new}})
\vec{p}_i(\vc^{\text{new}}), \ \ \ i = 1, 2, \ldots, n,
\end{equation}
where $\{\vec{p}_i(\vc^{\text{new}})\}_{i = 1}^n$ are the column
vectors of $P^{\text{new}}$.

Note that, (\ref{eq:solve_Pnew}) can be computed as follows. Compute
$H := (I - \frac{1}{2}Y) \tran{P}$ and let $\vec{h}_i$ be the $i$th
column of $H$ at first. Then, solve $\vec{w}_i, i = 1, 2, \ldots,
n$, iteratively from the $n$ linear systems:
\begin{equation}
\label{eq:solve_wi} (I + \frac{1}{2}Y)\vec{w}_i = \vec{h}_i, \ \ \ i
= 1, 2, \ldots, n.
\end{equation}
Finally, set $P^{\text{new}} := \tran{[\vec{w}_1, \ldots,
\vec{w}_n]}$. Since $P$ is an orthogonal matrix and $Y$ is
skew-symmetric matrix, we see that $P^{\text{new}}$ must be
orthogonal. To maintain the orthogonality of $P^{\text{new}}$,
(\ref{eq:solve_wi}) cannot be solved inexactly. One could expect
that it requires only few iterations to solve each system of
(\ref{eq:solve_wi}). This is due to the fact that, as $\{\vc^k\}$
converges to $\vc^*$, $\|Y_k\|$ converges to zero, see \cite[eq.
(3.64)]{Friedland1987}. Consequently, the coefficient matrix on the
left-hand side of (\ref{eq:solve_wi}) approaches the identity matrix
in the limit.

For solving the general nonlinear equation $f(\vec{x}) = \vec{0}$,
linesearch techniques \cite{Dennis1996} are often used to enlarge
the convergence basin of a locally convergent method. They are based
on a globally convergent method for a problem of the form
$\min\limits_{\vec{x} \in \RS^n} M(\vec{x})$, where $M$ is an
appropriately chosen merit function whose global minimum is a zero
of $f$. In these cases, for a given direction $\vec{s} \in \RS^n$,
we have the iteration form $\vec{x}_{k + 1} = \vec{x}_k + \alpha
\vec{s}$, where $\alpha \in (0, 1]$ is such that $M(\vec{x}_k +
\alpha \vec{s}) < M(\vec{x}_k)$. The existence of such an $\alpha$
is ensured if there exists an $\alpha_0 > 0$ such that $M(\vec{x}_k
+ \alpha \vec{s}) < M(\vec{x}_k)$ for all $\alpha < \alpha_0$.

In typical linesearch strategies, the step length $\alpha$ is chosen
by using so-called backtracking approach. Among the backtracking
method, inexact Newton backtracking method (INB)
\cite{Eisenstat1994} is a globally convergent process where the
$k$th iteration of an inexact Newton method is embedded in a
backtracking strategy. The merit function $M$ of INB is usually used
$M := \|f\|$, see for example
\cite{Eisenstat1994,Eisenstat1996,Pernice1998,Bellavia2001}. Thanks
to (\ref{fun:f(c)}), this will involve computing $\lambda_i(\vc^k)$
of $A(\vc^k)$ which are costly to compute. Our intention here is to
replace them by the Rayleigh quotient (see (\ref{comp:rhoinew})). In
Section 3 we will show that this replacement retain superlinear and
global convergence.

The details of our algorithm for solving the IEP are specified as
Algorithm \ref{al:GBIEP_Cayley}. In step 7, the following sufficient
decrease in the merit function $\|\vrho(\vc) - \vec{\lambda}^*\|$
based on the Rayleigh quotient is provided:
$$
\|\vrho(\vc^k + \vdc^k) - \vec{\lambda}^*\| \leq (1 - \xi(1 -
\eta_k)) \|\vrho(\vc^k) - \vec{\lambda}^*\|.
$$
The while loop in step 7 is also called backtracking loop below.

\begin{algorithm}
\caption{Inexact Newton Backtracking Cayley Transform Method for
IEP} \label{al:GBIEP_Cayley} For any $\vec{c}^0 \in \RS^n,
\eta_{\text{max}} \in [0,1), \xi \in (0,1), 0 < \theta_{min} <
\theta_{max} < 1$. Compute the orthonormal eigenvectors
$\{\vec{q}_i(\vec{c}^0)\}_{i = 1}^n$ and the eigenvalues
$\{\lambda_i(\vec{c}^0)\}_{i = 1}^n$ of $A(\vec{c}^0)$. Let
\begin{align}
P_0 & := [\vp_1(\vc^0), \vp_2(\vc^0), \ldots, \vp_n(\vc^0)] =
[\vq_1(\vc^0), \vq_2(\vc^0), \ldots, \vq_n(\vc^0)], \label{mat:P0}\\
\vec{\rho}(\vc^0) & := \tran{[\rho_1(\vc^0), \rho_2(\vc^0), \ldots,
\rho_n(\vc^0)]} = \tran{[\lambda_1(\vc^0), \lambda_2(\vc^0), \ldots,
\lambda_n(\vc^0)]}. \label{vec:rhoc0}
\end{align}
For $k = 0, 1, 2, \ldots$ until convergence do:
\newcounter{newlist}
\begin{list}
{step \arabic{newlist}.}
{\usecounter{newlist}
\setlength{\itemsep}{0em}
\setlength{\leftmargin}{1em}
\setlength{\rightmargin}{1em}}
\item
Form $[J_k]_{ij} = \tran{\vp_i(\vc^k)} A_j \vp_i(\vc^k)$ for $1 \leq
i, j \leq n$.
\item
Solve $\Delta\overline{\vc}^k$ inexactly from the approximate
Jacobian equation:
\begin{equation}
\label{eq:appr_Jaco_eq} J_k \Delta\overline{\vc}^k + \vrho(\vc^k)-
\vec{\lambda}^* = \vec{0}
\end{equation}
such that
\begin{equation}
\label{cond:res_cond} \|J_k \Delta\overline{\vc}^k + \vrho(\vc^k)-
\vec{\lambda}^*\| \leq \overline{\eta}_k \|\vrho(\vc^k) -
\vec{\lambda}^*\|,
\end{equation}
where $\overline{\eta}_0 \in (0,1)$ and
\begin{equation}
\label{cons:oetak} \overline{\eta}_k :=
\min\left\{\frac{\|\vrho(\vc^k) -
\vec{\lambda}^*\|^{\beta}}{\|\vec{\lambda}^*\|^\beta},
\frac{\|\vrho(\vc^k) - \vec{\lambda}^*\|^{\beta}}{\|\vrho(\vc^{k -
1}) - \vec{\lambda}^*\|^{\beta}}, \eta_{\text{max}}\right\}, \ \ \
\beta \in (1, 2], \ k = 1, 2, \ldots.
\end{equation}
\item
Set $\vdc^k = \Delta\overline{\vc}^k$ and $\eta_k =
\overline{\eta}_k$.
\item
Form the skew-symmetric matrix $\overline{Y}_k$ by (\ref{comp:Yij})
with $\vc^{\text{new}} = \vc^k + \vdc^k$.
\item
Compute matrix $\overline{P}_k := [\vp_1(\vc^k + \vdc^k),
 \ldots, \vp_n(\vc^k + \vdc^k)]$ by solving
(\ref{eq:solve_Pnew}).
\item
Compute $\vrho(\vc^k + \vdc^k) = \tran{(\rho_1(\vc^k + \vdc^k),
 \ldots, \rho_n(\vc^k + \vdc^k))}$ by
(\ref{comp:rhoinew}).
\item
While $\|\vrho(\vc^k + \vdc^k) - \vec{\lambda}^*\|
> (1 - \xi(1 - \eta_k)) \|\vrho(\vc^k) - \vec{\lambda}^*\|$ do:
$$
\text{choose\ } \theta \in [\theta_{\min}, \theta_{\max}], \text{\
then update } \vdc^k \leftarrow \theta \vdc^k \text{ and } \eta_k
\leftarrow 1 - \theta(1 - \eta_k).
$$

\item
Set $\vc^{k + 1} = \vc^k + \vdc^k$. As steps 4-6, compute,
respectively, the new approximations $P_{k + 1}:= [\vp_1(\vc^{k +
1}), \ldots, \vp_n(\vc^{k + 1})]$ and $\vrho(\vc^{k + 1}) :=
\tran{(\rho_1(\vc^{k + 1}), \ldots, \rho_n(\vc^{k + 1}))}$.
\end{list}
\end{algorithm}

\section{Convergence Analysis}

In this section, we analyze the global behavior of Algorithm
\ref{al:GBIEP_Cayley}.
We will show that if the given eigenvalues are distinct and if there
exists an accumulation point $\vc^*$ of $\{\vc^k\}$ such that the
Jacobian matrix $J(\vc^*)$ is invertible, then the iterations are
guaranteed to remain near $\vc^*$ and $\vrho(\vc^*) -
\vec{\lambda}^* = \vec{0}, \vc^k \to \vc^*$ as $k \to \infty$.
Furthermore, for $k$ sufficiently large, we have the equality
$\vc^{k+1} = \vc^k + \Delta\overline{\vc}^k$. Thus, we obtain that
the ultimate rate of convergence is $\beta$ which depends on the
choices of the $\overline{\eta}_k$ given in (\ref{cons:oetak}) .

It is worth noting that, if $\{\vc^k\}$ has no accumulation point,
or $\{\vc^k\}$ has one or more accumulation points and the Jacobian
matrix is singular at each of them, or the vector
$\Delta\overline{\vc}^k$ computed by solving the Jacobian equation
(\ref{eq:appr_Jaco_eq}) is such that $\Delta\overline{\vc}^k =
\vec{0}$, then our algorithm fails.

It is clear that, if $\vrho(\vc^k) - \vec{\lambda}^* \to 0$ as $k
\to \infty$ and $\vc^*$ is an accumulation point of $\{\vc^k\}$,
then $\rho(\vc^*) - \vec{\lambda}^* = 0$. Let $\{P_k\}$ be generated
by Algorithm \ref{al:GBIEP_Cayley} (see step 7) and define $E_k :=
P_k - Q^*$ for each $k = 0,1,\ldots$. The following lemma is taken
from \cite[Lamma 2]{BaiMoriniXu2007}.

\begin{lemma}[\cite{BaiMoriniXu2007}]
\label{lem:A(c)_Lip} For any $\vc, \overline{\vc} \in \RS^n$, we
have
$$
\|A(\vc) - A(\overline{\vc})\| \leq L \|\vc - \overline{\vc}\|,
$$
where $L := \left(\sum_{i = 1}^n \|A_i\|^2\right)^{\frac{1}{2}}$.
\end{lemma}

Based on Lemma \ref{lem:A(c)_Lip}, the following lemma is a
straightforward application of \cite[Lemma 4]{BaiChanMorini2004}.

\begin{lemma}
\label{lem:est_norm_rho(ck)-lambda*} Assume that $\vc^*$ is an
accumulation point of $\{\vc^k\}$. Let the given eigenvalues
$\{\lambda_i^*\}_{i = 1}^n$ be distinct. Then
$$
\|\vrho(\vc^k) - \vec{\lambda}^*\| \leq \sqrt{n}L \|\vc^k - \vc^*\|
+ \mu_2 \|E_k\|,
$$
for any $k \in \NS$, where $\mu_2 := 2 \sqrt{n} \cdot \|A(\vc^*)\|$.
\end{lemma}


As shown in \cite[Thoerem 2.3]{Sun1985}, in the case when the given
eigenvalues $\{\lambda_i^*\}_{i = 1}^n$ are distinct, the
eigenvalues of $A(\vc)$ are distinct too for any point $\vc$ in some
neighborhood of $\vc^*$. It follows that the function $f(\cdot)$
defined in (\ref{fun:f(c)}) is analytic in the same neighborhood.
However, if $\vc$ is not near the solution $\vc^*$, the analyticity
of the function $f(\cdot)$ cannot be guaranteed.

For any symmetric matrix $X \in \RS^{n \times n}$, set $\Lambda(X)
:= \diag(\lambda_1(X), \ldots, \lambda_n(X))$, where $\lambda_i(X),
i = 1, 2, \ldots, n,$ are the eigenvalues of $X$. As proved by Sun
and Sun \cite[Theorem 4.7]{SunSun2003}, $\Lambda(\cdot)$ is a
strongly semismooth function. Based on this result, for any $\vc \in
\RS^n$ and $\varepsilon > 0$, there exists $\delta > 0$ sufficiently
small such that
\begin{equation*}
\|\vrho(\overline{\vc}) - \vrho(\vc) - J(\vc) (\overline{\vc} -
\vc)\| \leq \varepsilon \|\overline{\vc} - \vc\|^2, \ \ \ \text{for
any } \overline{\vc} \in \ball(\vc, \delta).
\end{equation*}

The following lemma says that, the backtracking loop in step 7 of
Algorithm \ref{al:GBIEP_Cayley} terminates after a finite number of
steps.

\begin{lemma}
\label{lem:while_terminate_1} There exists $\eta_k^{\min} \in [0,1)$
such that, for any $\eta_k \in [\eta_k^{\min}, 1)$, there is
$\vdc^k$ satisfying
$$
\|\vrho(\vc^k + \vdc^k) - \vec{\lambda}^*\| \leq [1 - \xi(1 -
\eta_k)] \|\vrho(\vc^k) - \vec{\lambda}^*\|.
$$
\end{lemma}

\begin{proof}
By using the strong semismoothness of all eigenvalues of a real
symmetric matrix \cite{SunSun2003}, for any given $\varepsilon > 0$,
there exists $\delta > 0$ sufficiently small such that
\begin{equation}
\label{leq:stronglysemismooth} \|\vrho(\vc^k + \vdc) - \vrho(\vc^k)
- J_k \vdc\| \leq \varepsilon \|\vdc\|
\end{equation}
whenever $\|\vdc\| \leq \delta$. Choose
\begin{equation}
\label{cons:varepsilon} \varepsilon = \frac{(1 - \xi)(1 -
\overline{\eta}_k)\|\vrho(\vc^k) -
\vec{\lambda}^*\|}{\|\Delta\overline{\vc}^k\|}
\end{equation}
and set
\begin{equation}
\label{cons:etakmin} \eta_k^{\min} := \max\left\{\overline{\eta}_k,
1 - \frac{(1 -
\overline{\eta}_k)\delta}{\|\Delta\overline{\vc}^k\|}\right\}.
\end{equation}
For any $\eta_k \in [\eta_k^{\min}, 1)$, let $\vdc^k := \frac{1 -
\eta_k}{1 - \overline{\eta}_k} \Delta\overline{\vc}^k$. Then, by the
definition of $\eta_k^{\min}$ given in (\ref{cons:etakmin}), we have
$$
\|\vdc^k\| = \frac{1 - \eta_k}{1 - \overline{\eta}_k}
\|\Delta\overline{\vc}^k\| \leq \frac{1 - \eta_k^{\min}}{1 -
\overline{\eta}_k} \|\Delta\overline{\vc}^k\| \leq \frac{(1 -
\overline{\eta}_k)\delta}{\|\Delta\overline{\vc}^k\|} \cdot
\frac{1}{1 - \overline{\eta}_k} \|\Delta\overline{\vc}^k\| = \delta.
$$
On the other hand, by (\ref{cond:res_cond}), one gets that
\begin{align*}
\|\vrho(\vc^k) - \vec{\lambda}^* + J_k \vdc^k\| & \leq \frac{\eta_k
- \overline{\eta}_k}{1 - \overline{\eta}_k} \|\vrho(\vc^k) -
\vec{\lambda}^*\| + \frac{1 - \eta_k}{1 - \overline{\eta}_k}
\|\vrho(\vc^k) - \vec{\lambda}^* + J_k \Delta\overline{\vc}^k\| \\
& \leq \frac{\eta_k - \overline{\eta}_k}{1 - \overline{\eta}_k}
\|\vrho(\vc^k) - \vec{\lambda}^*\| + \frac{1 - \eta_k}{1 -
\overline{\eta}_k} \cdot \overline{\eta}_k \|\vrho(\vc^k) - \vec{\lambda}^*\| \\
& = \eta_k \|\vrho(\vc^k) - \vec{\lambda}^*\|.
\end{align*}
This together with (\ref{leq:stronglysemismooth}) and
(\ref{cons:varepsilon}) yields that
\begin{align*}
\|\vrho(\vc^k + \vdc^k) - \vec{\lambda}^*\| & \leq \|\vrho(\vc^k +
\vdc^k) - \vrho(\vc^k) - J_k \vdc^k\| + \|\vrho(\vc^k) -
\vec{\lambda}^* + J_k \vdc^k\| \\
& \leq \varepsilon \|\vdc^k\| + \eta_k \|\vrho(\vc^k) -
\vec{\lambda}^*\| \\
& = \varepsilon \frac{1 - \eta_k}{1 - \overline{\eta}_k}
\|\Delta\overline{\vc}^k\| + \eta_k \|\vrho(\vc^k) -
\vec{\lambda}^*\| \\
& = (1 - \xi)(1 - \eta_k) \|\vrho(\vc^k) - \vec{\lambda}^*\| +
\eta_k \|\vrho(\vc^k) - \vec{\lambda}^*\| \\
& = [1 - \xi(1 - \eta_k)] \|\vrho(\vc^k) - \vec{\lambda}^*\|,
\end{align*}
and the proof is completed.
\end{proof}

Next, we give sufficient conditions for Algorithm
\ref{al:GBIEP_Cayley} not to break down in the backtracking loop in
step 7.

\begin{lemma}
\label{lem:while_terminate_2}  If $\vrho(\vc^k) - \vec{\lambda}^*
\neq 0$ and there exists $\Gamma > 0$ for which
\begin{equation}
\label{cond:norm_dck_1} \|\vdc^k\| \leq \Gamma (1 - \eta_k)
\|\vrho(\vc^k) - \vec{\lambda}^*\|,
\end{equation}
then, the backtracking loop terminates.
\end{lemma}

\begin{proof}
For constant $\Gamma$ in (\ref{cond:norm_dck_1}) and the given $\xi
\in (0,1)$, choosing $\varepsilon = \frac{1 - \xi}{\Gamma}$, there
exists $\delta > 0$ sufficiently small such that
(\ref{leq:stronglysemismooth}) holds whenever $\|\vdc\| \leq
\delta$. We choose $\eta_k \in [\overline{\eta}_k, 1]$ satisfying
\begin{equation}
\label{leq:1-etak} 1 - \eta_k < \frac{\delta}{\Gamma \|\vrho(\vc^k)
- \vec{\lambda}^*\|}.
\end{equation}
It follows that
$$
\|\vdc^k\| \leq \Gamma (1 - \eta_k) \|\vrho(\vc^k) -
\vec{\lambda}^*\| < \delta,
$$
which gives $\vc^k + \vdc^k \in \ball(\vc^k, \delta)$. Thus, we have
$$
\|\vrho(\vc + \vdc^k) - \vrho(\vc^k) - J_k\vdc^k\| \leq \frac{1 -
\xi}{\Gamma} \|\vdc^k\|.
$$
This together with (\ref{cond:res_cond}) gives that
\begin{align*}
\|\vrho(\vc^k + \vdc^k) - \vec{\lambda}^*\| & \leq \|\vrho(\vc^k) -
\vec{\lambda}^* + J_k \vdc^k\| + \|\vrho(\vc^k + \vdc^k) -
\vrho(\vc^k) - J_k\vdc^k\| \\
& \leq \eta_k \|\vrho(\vc^k) - \vec{\lambda}^*\| + (1 - \xi)(1 -
\eta_k) \|\vrho(\vc^k) - \vec{\lambda}^*\| \\
& = (1 - \xi(1 - \eta_k))\|\vrho(\vc^k) - \vec{\lambda}^*\|.
\end{align*}
It follows that the backtracking loop terminates. This completes the
proof.
\end{proof}

The next lemma gives condition under which (\ref{cond:norm_dck_1})
is satisfied.

\begin{lemma}
\label{lem:exist_Gamma}  Assume that the Jacobian matrix $J_k$ is
invertible and set $M_k := \|\inv{J}_k\|$. Then there exists
$\Gamma$ such that $(\ref{cond:norm_dck_1})$ holds.
\end{lemma}

\begin{proof}
By using condition (\ref{cond:res_cond}), one has that
\begin{align*}
\|\vdc^k\| & \leq \|\inv{J}_k\| \|J_k \vdc^k\| \\
& \leq M_k(\|\vrho(\vc^k) - \vec{\lambda}^*\| + \|\vrho(\vc^k) -
\vec{\lambda}^* + J_k \vdc^k\|) \\
& \leq M_k(1 + \eta_k)\|\vrho(\vc^k) - \vec{\lambda}^*\| \\
& = M_k \frac{1 + \eta_k}{1 - \eta_k} (1 - \eta_k) \|\vrho(\vc^k) -
\vec{\lambda}^*\| \\
& \leq M_k \frac{1 + \eta_{\max}}{1 - \eta_{\max}} (1 - \eta_k)
\|\vrho(\vc^k) - \vec{\lambda}^*\|.
\end{align*}
We finish the proof by taking $\Gamma := M_k \frac{1 +
\eta_{\max}}{1 - \eta_{\max}}$.
\end{proof}

Lemmas \ref{lem:while_terminate_2} and \ref{lem:exist_Gamma} yield
the result below.

\begin{corollary}
\label{cor:1-etak} Assume that $\vrho(\vc^k) - \vec{\lambda}^* \neq
\vec{0}$ and $J_k$ is invertible. Set $M_k := \|\inv{J}_k\|$ and
$\Gamma := M_k \frac{1 + \eta_{\max}}{1 - \eta_{\max}}$. Then, the
backtracking loop in step 6 of Algorithm $\ref{al:GBIEP_Cayley}$
terminates with
\begin{equation}
\label{geq:1-etak} 1 - \eta_k \geq \min\left\{1 - \overline{\eta}_k,
\frac{\delta \theta_{\textup{min}}}{\Gamma \|\vrho(\vc^k) -
\vec{\lambda}^*\|}\right\}
\end{equation}
for any $\delta > 0$ small enough such that
$(\ref{leq:stronglysemismooth})$ holds whenever $\|\vdc\| \leq
\delta$.
\end{corollary}

\begin{proof}
Suppose that $\eta_k$ is the final value determined by the
while-loop. If $\eta_k = \overline{\eta}_k$, then (\ref{geq:1-etak})
is trivial. Assume that $\eta_k \neq \overline{\eta}_k$, that is,
the body of the while-loop has been executed at  least once.
Denoting the penultimate value by $\eta_k^-$, then it follows from
(\ref{leq:1-etak}) that $1 - \eta_k^- \geq \frac{\delta}{\Gamma
\|\vrho(\vc^k) - \vec{\lambda}^*\|}$. Thus, we have
$$
1 - \eta_k = \theta (1 - \eta_k^-) \geq \theta_{\textup{min}} (1 -
\eta_k^-) \geq \frac{\delta \theta_{\textup{min}}}{\Gamma
\|\vrho(\vc^k) - \vec{\lambda}^*\|}.
$$
This completes the proof.
\end{proof}

\begin{lemma}
\label{lem:convergence1} Assume that $\vc^*$ is an accumulation
point of $\{\vc^k\}$ such that there exists a constant $\Gamma$
independent of $k$ for which
\begin{equation}
\label{cond:norm_dck_2} \|\vdc^k\| \leq \Gamma (1 - \eta_k)
\|\vrho(\vc^k) - \vec{\lambda}^*\|,
\end{equation}
whenever $\vc^k$ is sufficiently near $\vc^*$  and $k$ is sufficient
large. Then $\vc^k \to \vc^*$ as $k \to \infty$.
\end{lemma}

\begin{proof}
Suppose that $\vc^k \nrightarrow \vc^*$. Then, there exists $\delta
> 0$ sufficiently small such that there are infinitely many $k$ for
which $\vc^k \not\in \ball(\vc^*,\delta)$ and
(\ref{cond:norm_dck_2}) holds whenever $\vc^k \in
\ball(\vc^*,\delta)$ for $k$ sufficiently large.

Since $\vc^*$  is an accumulation point of $\{\vc^k\}$, there exists
subsequence $\{\vc^{k_i}\} \subset \{\vc^k\}$ such that $\vc^{k_i}
\in \ball(\vc^*, \delta/i)$ for $i$ sufficiently large. Choose
$\ell_i > 0$ satisfying $k_i + \ell_i < k_{i + 1}$ and $\vc^{k_i +
\ell_i} \not\in \ball(\vc^*, \delta)$. It follows that
$$
\|\vrho(\vc^{k_i + j - 1}) - \vec{\lambda}^*\| \leq \frac{1}{\Gamma
(1 - \eta_{k_i + j - 1})} \left(\|\vrho(\vc^{k_i + j - 1}) -
\vec{\lambda}^*\| - \|\vrho(\vc^{k_i + j}) -
\vec{\lambda}^*\|\right), \ \ \ j = 1, \ldots, \ell_i.
$$
Then, by (\ref{cond:norm_dck_2}), we have, for $i$ sufficiently
large,
\begin{align*}
\frac{\delta}{2} \leq \|\vc^{k_i + \ell_i} - \vc^{k_i}\| \leq
\sum_{k = k_i}^{k_i + \ell_i - 1} \|\vdc^k\| & \leq \sum_{k =
k_i}^{k_i + \ell_i - 1} \Gamma (1 - \eta_k)
\|\vrho(\vc^k) - \vec{\lambda}^*\| \\
& \leq \frac{\Gamma}{t} (\|\vrho(\vc^{k_i}) - \vec{\lambda}^*\| -
\|\vrho(\vc^{k_{i+1}}) - \vec{\lambda}^*\|).
\end{align*}
Note that $\vc^{k_i} \to \vc^*$ as $i \to \infty$. It follows that
$\|\vrho(\vc^{k_i}) - \vec{\lambda}^*\| - \|\vrho(\vc^{k_{i+1}}) -
\vec{\lambda}^*\| \to 0$, which is a contradiction. Therefore,
$\vc^k \to \vc^*$ as $k \to \infty$.
\end{proof}

\begin{lemma}
\label{lem:inv_Jk} Assume that $\vc^*$ is an accumulation point of
$\{\vc^k\}$ such that $J(\vc^*)$ is invertible. Set $M =
\|\inv{J(\vc^*)}\|$ and let $\delta_1$ be such that $0 < \delta_1
\leq \frac{1}{2 \mu_1 M}$, where $\mu_1 := 2n \cdot \max\limits_{1
\leq j \leq n} \|A_j\|$. Suppose that $\|E_k\| \leq \delta_1$ for
$k$ sufficiently large. Then, for all $k$ sufficiently large,
$$
\|\inv{J}_k\| \leq \frac{\|\inv{J(\vc^*)}\|}{1 - \mu_1 \delta_1
\|\inv{J(\vc^*)}\|} \leq 2 M.
$$
\end{lemma}

\begin{proof}
By the definitions of $[J_k]_{ij}$ and $[J(\vc^*)]_{ij}$, for all
$k$ sufficiently large,
\begin{align*}
|[J_k]_{ij} - [J(\vc^*)]_{ij}| & \leq |\tran{\vp_i(\vc^k)} A_j
\vp_i(\vc^k) - \tran{\vp_i(\vc^k)} A_j \vq_i(\vc^*)| +
|\tran{\vp_i(\vc^k)} A_j \vq_i(\vc^*) - \tran{\vq_i(\vc^*)} A_j
\vq_i(\vc^*)| \\
& \leq \|\tran{\vp_i(\vc^k)}\| \|A_j\| \|\vp_i(\vc^k) -
\vq_i(\vc^*)\| + \|\tran{\vp_i(\vc^k)} - \tran{\vq_i(\vc^*)}\|
\|A_j\| \|\vq_i(\vc^*)\| \\
& = 2\|A_j\| \|\vp_i(\vc^k) - \vq_i(\vc^*)\|, \ \ \ 1 \leq i, j \leq
n.
\end{align*}
Then, we have, for all $k$ sufficiently large,
$$
\|J_k - J(\vc^*)\| \leq \|J_k - J(\vc^*)\|_F \leq 2n \cdot \max_{1
\leq j \leq n} \|A_j\| \cdot \max_{1 \leq i \leq n} \|\vp_i(\vc^k) -
\vq_i(\vc^*)\|,
$$
Noting that, $\vp_i(\vc^k) - \vq_i(\vc^*)$ is the $i$th column of
$E_k$, then $\|\vp_i(\vc^k) - \vq_i(\vc^*)\| \leq \|E_k\|$ for $i =
1, \ldots, n$. So, for all $k$ sufficiently large,
$$
\|\inv{J(\vc^*)}\| \|J_k - J(\vc^*)\| \leq \mu_1 \delta_1 M \leq
\frac{1}{2} < 1.
$$
It follows from Banach lemma that $J_k$ is invertible and
$\|\inv{J}_k\| \leq 2M$ for all $k$ sufficiently large. This
completes the proof.
\end{proof}

Lemmas \ref{lem:exist_Gamma}, \ref{lem:convergence1} and
\ref{lem:inv_Jk} yield the result below.

\begin{corollary}
\label{cor:convergence2} Assume that $\vc^*$ is an accumulation
point of $\{\vc^k\}$  such that $J(\vc^*)$ is invertible. Set $M :=
\|\inv{J(\vc^*)}\|$ and let $\delta_1$ be determined by Lemma
$\ref{lem:inv_Jk}$. Assume that $\|E_k\| \leq \delta_1$ for $k$
sufficiently large. Then $\vc^k \to \vc^*$ as $k \to \infty$.
\end{corollary}

\begin{proof}
By Lemma \ref{lem:inv_Jk}, we know $J_k$ is invertible and
$\|\inv{J}_k\| \leq 2M$ for all $k$ sufficiently large. From the
proof of Lemma \ref{lem:exist_Gamma}, (\ref{cond:norm_dck_2}) holds
for the constant $\Gamma = 2 M \frac{1 + \eta_{\max}}{1 -
\eta_{\max}}$ independent of $k$. Therefore, $\vc^k \to \vc^*$ as $k
\to \infty$ follows from Lemma \ref{lem:convergence1}.
\end{proof}

\begin{corollary}
\label{cor:convergence3} Assume that $\vc^*$ is an accumulation
point of $\{\vc^k\}$ such that $J(\vc^*)$ is invertible. Then,
$\vrho(\vc^k) - \vec{\lambda}^* \to 0$ as $k \to \infty$. Moreover,
for all $k$ sufficiently large, we have $\eta_k =
\overline{\eta}_k$.
\end{corollary}

\begin{proof}
If $\vdc^k$ is computed in the backtracking loop, the backtracking
terminates with $\eta_k$ such that (\ref{geq:1-etak}) holds. Since
$\vc^k \to \vc^*$ by Corollary \ref{cor:convergence2}, we have
$\vc^k \in \ball(\vc^*, \delta)$ for all $k$ sufficiently large.
Thus, the series $\sum_{k = 0}^\infty (1 - \eta_k)$ is divergent.
Then, we have
$$
\prod_{1 \leq i < k} [1 - \xi(1 - \eta_i)] \leq \prod_{1 \leq i < k}
\me^{- \xi(1 - \eta_i)} = \me^{- \xi \cdot \sum\limits_{i = 1}^{k -
1} (1 - \eta_i)} \to 0, \ k \to \infty,
$$
and so,
$$
\|\vrho(\vc^k) - \vec{\lambda}^*\| \leq [1 - \xi(1 - \eta_{k - 1})]
\|\vrho(\vc^{k - 1}) - \vec{\lambda}^*\|   \leq \|\vrho(\vc^0) -
\vec{\lambda}^*\| \cdot \prod_{0 \leq i < k} [1 - \xi(1 - \eta_i)]
\to 0, \ k \to \infty.
$$
This together with (\ref{geq:1-etak}) yields  that $\eta_k =
\overline{\eta}_k$ for all $k$ sufficiently large.
\end{proof}

\begin{lemma}
\label{lem:est_norm_v} Assume that $\vc^*$ is an accumulation point
of $\{\vc^k\}$ such that $J(\vc^*)$ is invertible. Then, there
exists $\delta_2 > 0$ such that
\begin{equation}
\label{est:norm_v} \|\vrho(\vc^k) - \vec{\lambda}^* - J_k(\vc^k -
\vc^*)\| \leq \mu_3 \|E_k\|_F^2
\end{equation}
whenever $\vc^k \in \ball(\vc^*, \delta_2)$ and $k$ sufficiently
large, where $\mu_3 := 2 \max\limits_{1 \leq i \leq n}
|\lambda_i^*|$.
\end{lemma}

\begin{proof}
Set $H_k := \tran{Q(\vc^*)} P_k - \I$. Then, we have, for all $k$
sufficiently large,
$$
\tran{P}_k A(\vc^*) P_k = \tran{(\I + H_k)} \Lambda^* (\I + H_k) =
\Lambda^* + \Lambda^* H_k + \tran{H}_k\Lambda^* + \tran{H}_k
\Lambda^* H_k.
$$
It follows that, for all $k$ sufficiently large,
$$
\tran{\vp_i(\vc^k)}A(\vc^*)\vp_i(\vc^k) = \lambda_i^* + 2
\lambda_i^* [H_k]_{ii} + \sum_{j = 1}^n \lambda_j^* [H_k]_{ij}^2, \
\ \ i = 1, 2, \ldots, n.
$$
Set
\begin{equation}
\label{comp:vi} v_i := 2 \lambda_i^* [H_k]_{ii} + \sum_{j = 1}^n
\lambda_j^* [H_k]_{ij}^2, \ \ \ i = 1, 2, \ldots, n,
\end{equation}
and write $\vec{v} = \tran{(v_1, v_2, \ldots, v_n)}$. It follows
that $J_k\vc^* = \vec{\lambda}^* - \vec{a}^k + \vec{v}$, where
$$
[\vec{a}^k]_i = \tran{\vec{p}_i(\vc^k)} A_0 \vec{p}_i(\vc^k), \ \ \
i = 1, 2, \ldots, n.
$$
In view of that $J_k \vc^k = \vrho(\vc^k) - \vec{a}^k$ for any $k
\in \NS$, one has that, for all $k$ sufficiently large,
\begin{equation}
\label{vec:v} \vec{v} = \vrho(\vc^k) - \vec{\lambda}^* - J_k(\vc^k -
\vc^*).
\end{equation}
Since
$$
\I + H_k \tran{H}_k + H_k \tran{H}_k = (\I + H_k)\tran{(\I + H_k)} =
\tran{Q(\vc^*)} P_k \tran{P}_k Q(\vc^*) = \I,
$$
we have
\begin{equation}
\label{comp:Hkii} [H_k]_{ii} = - \frac{1}{2} \sum_{j = 1}^n
[H_k]_{ij}^2, \ \ \ i = 1, 2, \ldots, n.
\end{equation}
It follows from (\ref{comp:vi}) that, for all $k$ sufficiently
large,
\begin{equation}
\label{comp:vi2} v_i^2 = \left(2 \lambda_i^* [H_k]_{ii} + \sum_{j =
1}^n \lambda_j^* [H_k]_{ij}^2\right)^2 \leq 2 \left[4
(\lambda_i^*)^2 [H_k]_{ii}^2 + \left(\sum_{j = 1}^n \lambda_j^*
[H_k]_{ij}^2\right)^2\right].
\end{equation}
Combining (\ref{comp:Hkii}) with (\ref{comp:vi2}), one has that, for
all $k$ sufficiently large,
\begin{align*}
\sum_{i = 1}^n v_i^2 & \leq 2 \left[\sum_{i = 1}^n 4
(\lambda_i^*)^2[H_k]_{ii}^2 + \sum_{i = 1}^n \left(\sum_{j = 1}^n
\lambda_j^* [H_k]_{ij}^2\right)^2 \right] \\
& = 2 \left[\sum_{i = 1}^n (\lambda_i^*)^2 \left(\sum_{j = 1}^n
[H_k]_{ij}^2\right)^2 + \sum_{i = 1}^n \left(\sum_{j = 1}^n
\lambda_j^* [H_k]_{ij}^2\right)^2 \right]\\
& \leq 2 \left[\max_{1 \leq i \leq n}|\lambda_i^*|^2 \cdot \sum_{i =
1}^n \left(\sum_{j = 1}^n [H_k]_{ij}^2\right)^2 + \max_{1 \leq i
\leq n}|\lambda_i^*|^2 \cdot \sum_{i = 1}^n \left(\sum_{j = 1}^n
[H_k]_{ij}^2\right)^2 \right] \\
& \leq 4 \max_{1 \leq i \leq n}|\lambda_i^*|^2 \cdot \left(\sum_{i =
1}^n
\sum_{j = 1}^n [H_k]_{ij}^2 \right)^2 \\
& = 4 \max_{1 \leq i \leq n}|\lambda_i^*|^2 \cdot \|H_k\|_F^4,
\end{align*}
which gives
$$
\|\vec{v}\| \leq 2 \max_{1 \leq i \leq n}|\lambda_i^*| \cdot
\|H_k\|_F^2.
$$
Note that,
$$
\|H_k\|_F = \|\tran{Q(\vc^*)}P_k - \I\|_F = \|P_k - Q(\vc^*)\|_F =
\|E_k\|_F.
$$
Therefore, we obtain from (\ref{vec:v}) that (\ref{est:norm_v})
holds for all $k$ sufficiently large. This completes the proof.
\end{proof}

\begin{lemma}
\label{lem:est_norm_ck+1-c*} Assume that $\vc^*$ is an accumulation
point of $\{\vc^k\}$ such that $J(\vc^*)$ is invertible. Then, there
exist $\delta_3 > 0, \delta_4 \in (0, \delta_1]$ sufficiently small
such that $\|E_k\| \leq \delta_4$ and
\begin{equation}
\label{est:norm_ck+1-c*} \|\vc^{k + 1} - \vc^*\| \leq 2M\left(\mu_3
\|E_k\|_F^2 + \left(\frac{2}{\mu_3}\right)^\beta \|\vrho(\vc^k) -
\vec{\lambda}^*\|^\beta\right)
\end{equation}
whenever $\vc^k \in \ball(\vc^*, \delta_3)$ and $k$ sufficiently
large, where $\delta_1$ is determined by Lemma $\ref{lem:inv_Jk}$.
\end{lemma}

\begin{proof}
Thanks to Corollary \ref{cor:convergence3}, we have $\eta_k =
\overline{\eta}_k$  and $\vdc^k = \Delta\overline{\vc}^k$ for all
$k$ sufficiently large. Thus, the Jacobian equation
(\ref{eq:appr_Jaco_eq}) is equivalent to, for all $k$ sufficiently
large,
$$
J_k \vc^{k + 1} - \vec{\lambda}^* + \vec{a}^k = \vec{0}.
$$
Assume that the  residual of this approximate Jacobian equation is
defined by $\vec{r}^k$, i.e., for all $k$ sufficiently large,
$$
\vec{r}^k = J_k \vc^{k + 1} - \vec{\lambda}^* + \vec{a}^k.
$$
This together with $J_k \vc^* - \vec{\lambda}^* + \vec{a}^k =
\vec{v}$ gives, for all $k$ sufficiently large,
$$
J_k(\vc^* - \vc^{k + 1}) = \vec{v} - \vec{r}^k.
$$
By (\ref{cond:res_cond}), Lemmas \ref{lem:inv_Jk} and
\ref{lem:est_norm_v}, we obtain, for all $k$ sufficiently large,
\begin{align*}
\|\vc^{k + 1} - \vc^*\| & \leq \|\inv{J}_k\| (\|\vec{v}\| +
\|\vec{r}^k\|) \\
& \leq 2M\left(\mu_3 \|E_k\|_F^2 + \frac{\|\vrho(\vc^k) -
\vec{\lambda}^*\|^{\beta + 1}}{\|\vec{\lambda}^*\|^\beta}\right)
\end{align*}
It follows from Lemma \ref{lem:est_norm_rho(ck)-lambda*} that
$$
\|\vrho(\vc^k) - \vec{\lambda}^*\| \leq \sqrt{n}L \|\vc^k - \vc^*\|
+ \mu_2 \|E_k\|, \ \ \ \forall \ k \in \NS.
$$
Thus, we can choose $\delta_3 > 0$ and $0 < \delta_4 \leq \delta_1$
sufficiently small such that
$$
\sqrt{n} L \delta_3 + \mu_2 \delta_4 \leq 1,
$$
whenever $\|\vc^k - \vc^*\| \leq \delta_3$ and $\|E_k\|_F \leq
\delta_4$. Therefore, combining this with the definition of $\mu_3$,
(\ref{est:norm_ck+1-c*}) follows.
\end{proof}

\begin{lemma}[\cite{Friedland1987}]
\label{lem:est_norm_X} There exist two positive numbers $\delta_5$
and $\omega_1$ such that, for any orthogonal matrix $P$ with $\|P -
Q(\vc^*)\| < \delta_5$, the skew-symmetric $X$ defined by $\me^X :=
\tran{P}Q(\vc^*)$ satisfies $\|X\| \leq \omega_1 \|P - Q(\vc^*)\|$.
\end{lemma}

Based on Lemma \ref{lem:est_norm_X}, by using the similar arguments
in the proof of \cite[Lemma 3.4]{ShenLi2012}, we can obtain the
following lemma. If we write $\me^{X_k} := \tran{P}_k Q^*$, then
there exists $C > 0$ such that $\|Y_k\| \leq C(\|\vc^{k + 1} -
\vc^*\| + \|X_k\|)$.

\begin{lemma}
\label{lem:est_norm_Ek+1} Suppose that the given eigenvalues
$\{\lambda_i^*\}_{i = 1}^n$ are distinct and the Jacobian matrix
$J(\vc^*)$ is invertible. Then, there exist $\omega_2 > 0$ and $0 <
\delta_6 < \min\{\delta_5, \frac{1}{(1 + \omega_1)C}\}$ such that,
for $k$ sufficiently large, if $\|\vc^{k + 1} - \vc^*\| \leq
\delta_6$ and $\|E_k\| \leq \delta_6$, then
$$
\|E_{k + 1}\| \leq \omega_2 (\|\vc^{k + 1} - \vc^*\| + \|E_k\|^2),
$$
where $\omega_1$ is determined by Lemma $\ref{lem:est_norm_X}$.
\end{lemma}

In order to prove our global convergence result for Algorithm
\ref{al:GBIEP_Cayley}, we introduce some notations. Let $\zeta_1 :=
\frac{2(\sqrt{n} L + \mu_2)}{\mu_3}$ and $\zeta_2 := 2M(\mu_3 +
\zeta_1^\beta)$. Set
\begin{align}
\tau & := \min\left\{1, \left(2M(\mu_3 + \zeta^\beta_1)\right)^{-
\frac{1}{\beta - 1}}, (\omega_2(\zeta_2 + 1))^{-
\frac{1}{\beta - 1}}\right\}, \label{cons:tau}\\
\delta & := \min\{\tau, \delta_2, \delta_3, \delta_4, \delta_6\}
\label{cons:delta}.
\end{align}
Our main global convergence result is as follows.

\begin{theorem}
\label{th:GlobalConvergence} Assume that $\{\vc^k\}$ is generated by
Algorithm $\ref{al:GBIEP_Cayley}$. Suppose that the given
eigenvalues $\{\lambda_i^*\}_{i = 1}^n$ are distinct and $\vc^*$ is
an accumulation point of $\{\vc^k\}$ such that $J(\vc^*)$ is
invertible. Then, $\vrho(\vc^k) - \vec{\lambda}^* \to \vec{0}$ and
$\vc^k \to \vc^*$ as $k \to \infty$. Moreover, the convergence is of
R-order $\beta$.
\end{theorem}

\begin{proof}
It follows immediately from Corollaries \ref{cor:convergence2} and
\ref{cor:convergence3} that $\vrho(\vc^k) - \vec{\lambda}^* \to
\vec{0}$ and $\vc^k \to \vc^*$ as $k \to \infty$. For the $\tau$
given in (\ref{cons:tau}), there exists $k_0$ sufficiently large
such that $\|\vc^{k_0} - \vc^*\| \leq \tau$ and $\|E_{k_0}\| \leq
\tau$. Set $\gamma := \delta/\tau$, where $\delta$ and $\tau$ are
given in (\ref{cons:delta}) and (\ref{cons:tau}), respectively.
Then, $\gamma \leq 1$. We will show that, for all $k \geq k_0$
sufficiently large,
\begin{align}
\|\vc^k - \vc^*\| & \leq \tau \cdot \gamma^{\beta^{k - k_0}},
\label{est:norm_ck-c*} \\
\|E_k\| & \leq \tau \cdot \gamma^{\beta^{k - k_0}}.
\label{est:norm_Ek}
\end{align}
Suppose that (\ref{est:norm_ck-c*}) and (\ref{est:norm_Ek}) hold for
some $k = \ell \geq k_0$. Consider the case $k = \ell + 1$. Thanks
to Lemma \ref{lem:est_norm_rho(ck)-lambda*}, we have, for all $\ell
\geq k_0$,
\begin{align}
\|\vrho(\vc^\ell) - \vec{\lambda}^*\| & \leq \sqrt{n} L \|\vc^\ell -
\vc^*\| + \mu_2 \|E_\ell\| \nonumber\\
& \leq  \sqrt{n} L \tau \cdot \gamma^{\beta^{\ell - k_0}} + \mu_2
\tau
\cdot \gamma^{\beta^{\ell - k_0}} \nonumber\\
& = (\sqrt{n} L + \mu_2) \tau \cdot \gamma^{\beta^{\ell - k_0}}.
\label{est:norm_vrho-lambda*}
\end{align}
Then, by using Lemma \ref{lem:est_norm_ck+1-c*}, one has that, for
all $\ell \geq k_0$,
\begin{align*}
\|\vc^{\ell + 1} - \vc^*\| & \leq 2M \left(\mu_3 \|E_\ell\|_F^2 +
\left(\frac{2}{\mu_3}\right)^\beta \|\vrho(\vc^\ell) -
\vec{\lambda}^*\|\right) \\
& \leq 2M\mu_3 \tau^2 \left(\gamma^{\beta^{\ell - k_0}}\right)^2 +
2M \left(\frac{2}{\mu_3}\right)^\beta (\sqrt{n}L + \mu_3)^\beta
\tau^\beta \gamma^{\beta^{\ell - k_0 + 1}} \\
& \leq 2M\left[\mu_3 + \left(\frac{2(\sqrt{n}L +
\mu_2)}{\mu_3}\right)^\beta\right] \tau^\beta \gamma^{\beta^{\ell -
k_0
+ 1}} \\
& \leq \tau \gamma^{\beta^{\ell - k_0 + 1}},
\end{align*}
where the last inequality follows from the definition of $\tau$ in
(\ref{cons:tau}). By Lemma \ref{lem:est_norm_Ek+1}, we have, for all
$\ell \geq k_0$,
\begin{align*}
\|E_{\ell + 1}\| & \leq \omega_2 (\|\vc^{\ell + 1} - \vc^*\| +
\|E_\ell\|^2)
\\
& \leq 2M\omega_2 (\mu_3 + \zeta_1^\beta) \tau^\beta
\gamma^{\beta^{\ell - k_0 + 1}} + \omega_2 \tau^2
\left(\gamma^{\beta^{\ell - k_0}}\right)^2 \\
& \leq \omega(2M(\mu_3 + \zeta_1^\beta) + 1) \tau^\beta \cdot
\gamma^{\beta^{\ell - k_0 + 1}} \\
& = \omega_2 (\zeta_2 + 1)\tau^\beta \cdot \gamma^{\beta^{\ell - k_0
+
1}} \\
& \leq \tau \cdot \gamma^{\beta^{\ell - k_0 + 1}}.
\end{align*}
Therefore, we conclude that (\ref{est:norm_ck-c*}) and
(\ref{est:norm_Ek}) hold for all $k \geq k_0$. Moreover, we see from
(\ref{est:norm_ck-c*}) that $\vc^k$ converges to $\vc^*$ with
R-order $\beta$. This completes the proof.
\end{proof}

If $\{\vc^k\}$ generated by Algorithm \ref{al:GBIEP_Cayley}
converges to a solution at which the Jacobian matrix is invertible,
then the ultimate rate of convergence is governed by the choices of
the $\eta_k (k = 0, 1, \ldots)$ as in the local theory of
\cite{BaiChanMorini2004}.

\section{A Modified Global Inexact Newton-Type Method}

We observe from Lemmas \ref{lem:while_terminate_2} and
\ref{lem:exist_Gamma} that, when the Jacobian matrix $J_k$ is
invertible, the condition (\ref{cond:norm_dck_1}) holds for some
$\vdc^k$, and so the backtracking loop terminates after a finite
number of steps. Comparing the condition (\ref{cond:norm_dck_1})
with the condition in the while loop, that is:
$$
\|\vrho(\vc^k + \vdc^k) - \vec{\lambda}^*\| \leq (1 - \xi(1 -
\eta_k)) \|\vrho(\vc^k) - \vec{\lambda}^*\|,
$$
one finds that the condition (\ref{cond:norm_dck_1}) is more
feasible in practice. To improve the practical effectiveness of
Algorithm \ref{al:GBIEP_Cayley}, we give a modified algorithm as
follows.

\begin{algorithm}[t]
\caption{Modified Inexact Newton Backtracking Cayley Transform
Method for IEP} \label{al:GBIEP_Cayley_mod} For any $\vec{c}^0 \in
\RS^n, \eta_{\text{max}} \in [0,1), \xi \in (0,1), 0 < \theta_{min}
< \theta_{max} < 1$. Compute the orthonormal eigenvectors
$\{\vec{q}_i(\vec{c}^0)\}_{i = 1}^n$ and the eigenvalues
$\{\lambda_i(\vec{c}^0)\}_{i = 1}^n$ of $A(\vec{c}^0)$. Let $P_0$
and $\vec{\rho}(\vc^0)$ be defined in (\ref{mat:P0}) and
(\ref{vec:rhoc0}), respectively. For $k = 0, 1, 2, \ldots$ until
convergence do:
\begin{list}
{step \arabic{newlist}.} {\usecounter{newlist}
\setlength{\itemsep}{0em} \setlength{\leftmargin}{1em}
\setlength{\rightmargin}{1em}}
\item
Form $[J_k]_{ij} = \tran{\vp_i(\vc^k)} A_j \vp_i(\vc^k)$ for $1 \leq
i, j \leq n$.
\item
If $J_k$ is singular, perform the same steps as in Algorithm
\ref{al:GBIEP_Cayley}; else perform
\begin{list}
{step 2.\arabic{newlist}.} {\usecounter{newlist}
\setlength{\itemsep}{0em} \setlength{\leftmargin}{1em}
\setlength{\rightmargin}{1em}}
\item
Solve (\ref{eq:appr_Jaco_eq}) for $\vdc^k$ such that condition
(\ref{cond:res_cond}) is satisfied.
\item
Set $\Gamma_k := M_k \frac{1 + \eta_{\text{max}}}{1 -
\eta_{\text{max}}}$, where $ M_k := \|\inv{J}_k\|$. Perform the
backtracking loop: \\ while $\|\vdc^k\|
> \Gamma_k(1 - \eta_k) \|\vrho(\vc^k) - \vec{\lambda}^*\|$ do:
$$
\text{choose\ } \theta \in [\theta_{\min}, \theta_{\max}], \text{\
then update } \vdc^k \leftarrow \theta \vdc^k \text{ and } \eta_k
\leftarrow 1 - \theta(1 - \eta_k).
$$
\item
The same as step 7 in Algorithm \ref{al:GBIEP_Cayley}.
\end{list}
\end{list}
\end{algorithm}

In our numerical examples, we will report the numerical performance
of Algorithm \ref{al:GBIEP_Cayley_mod} instead of Algorithm
\ref{al:GBIEP_Cayley}, since Algorithm \ref{al:GBIEP_Cayley_mod} is
more effective and feasible in practice.

\section{Numerical Examples}

In this section, we illustrate the effectiveness of Algorithm
\ref{al:GBIEP_Cayley_mod} for solving IEP on two examples. The tests
were carried out in MATLAB 7.0.

The given parameters used in our algorithm were $\eta_0 = 0.5,
\eta_{\max} = 0.9, \xi = 10^{-4}, \theta_{\min} = 0.1$ and
$\theta_{\max} = 0.9$. In the while loop, we choose $\theta \in
[\theta_{\min}, \theta_{\max}]$ to minimize $\|\vrho(\vc^k + \theta
\vdc^k) - \vec{\lambda}^*\|$ if 80 iterations of the backtracking
loop fail to produce the sufficient decrease in $\|\vrho(\vc) -
\vec{\lambda}^*\|$.

Linear systems (\ref{eq:solve_wi}) and (\ref{eq:appr_Jaco_eq}) are
all solved iteratively by the QMR method \cite{Freund1991} using the
MATLAB \textsf{qmr} function.  In order to guarantee the
orthogonality of $P^{\text{new}}$ in (\ref{eq:solve_wi}), this
system is solved up to  mathine precision \textsf{eps} ($\approx 2.2
\times 10^{-16}$). The stopping tolerances for system
(\ref{eq:appr_Jaco_eq}) is given as in (\ref{cons:oetak}).
The inner loop stopping tolerance for (\ref{eq:appr_Jaco_eq}) is
given by (\ref{cons:oetak}). The stopping criterion of the outer
iteration in our algorithm is
\begin{equation}
\label{eq:error} \|\tran{P}_k A(\vc^k)P_k - \Lambda_*\|_F \leq
10^{-10}.
\end{equation}

EXAMPLE 1. This is an inverse Toeplitz eigenvalue problem (see
\cite{Trench1997} for more detail on this inverse problem) with
distinct eigenvalues. The basis matrices $\{A_i\}_{i = 1}^5$ are
given as follows:
\begin{align*}
A_1 & = \left[\begin{array}{rrrrr} 1 & 0 & 0 & 0 & 0 \\
0 & 1 & 0 & 0 & 0 \\
0 & 0 & 1 & 0 & 0 \\
0 & 0 & 0 & 1 & 0 \\
0 & 0 & 0 & 0 & 1
\end{array}\right], \ \ \
A_2 = \left[\begin{array}{ccccc} 0 & 1 & 0 & 0 & 0 \\
1 & 0 & 1 & 0 & 0 \\
0 & 1 & 0 & 1 & 0 \\
0 & 0 & 1 & 0 & 1 \\
0 & 0 & 0 & 1 & 0
\end{array}\right], \ \ \
A_3 = \left[\begin{array}{ccccc} 0 & 0 & 1 & 0 & 0 \\
0 & 0 & 0 & 1 & 0 \\
1 & 0 & 0 & 0 & 1 \\
0 & 1 & 0 & 0 & 0 \\
0 & 0 & 1 & 0 & 0
\end{array}\right],\\
A_4 & = \left[\begin{array}{rrrrr} 0 & 0 & 0 & 1 & 0 \\
0 & 0 & 0 & 0 & 1 \\
0 & 0 & 0 & 0 & 0 \\
1 & 0 & 0 & 0 & 0 \\
0 & 1 & 0 & 0 & 0
\end{array}\right], \ \ \
A_5 = \left[\begin{array}{ccccc} 0 & 0 & 0 & 0 & 1 \\
0 & 0 & 0 & 0 & 0 \\
0 & 0 & 0 & 0 & 0 \\
0 & 0 & 0 & 0 & 0 \\
1 & 0 & 0 & 0 & 0
\end{array}\right].
\end{align*}
The given real eigenvalues and a solution, respectively, are
$$
\vec{\lambda}^* = \tran{(-5.2361, -1.5876, -0.7639, -0.5555,
18.1431)} \  \ \text{and} \  \ \vc^* = \tran{(2, 3, 4, 5, 6)}.
$$
In Table \ref{tab:example1}, we report our numerical results for
various starting points:
\begin{align*}
(\text{a}) \ \ \vc^0 & = \tran{(1,2,3,4,5)};\ \ \ (\text{b}) \ \
\vc^0  = \tran{(1, 5, 10, 15, 20)};\ \
(\text{c}) \ \ \vc^0  = \tran{(11,12,13,14,15)};\\
(\text{d}) \ \ \vc^0 & = \tran{(21, 38, 46, 63, 81)}; \ \ (\text{e})
\ \ \vc^0  = \tran{(101, 112, 123, 134, 145)},
\end{align*}
where $\vc^0$, errs, ite. and $\vc^*$ stand for the starting point,
the error value of the left hand side of (\ref{eq:error}) for the
last three iterates of the algorithm, the number of outer iteration
and the accumulation point corresponding to the starting point.

\begin{table}[t]
\begin{center}
\begin{tabular}{l@{\hspace{1.2cm}}l@{\hspace{1.2cm}}l@{\hspace{1.8cm}}l@{\hspace{1.8cm}}l@{\hspace{0.3cm}}}
\hline
$\vc^0$ &  & $\beta = 1.5 $  & $\beta = 1.8 $  & $\beta = 2.0 $ \\
\hline
(a) & errs & 1.0659e-02  & 1.0659e-02  & 1.0659e-02\\
 &  & 1.1323e-05  & 1.1323e-05  & 1.1323e-05 \\
 &  & 8.0567e-12  & 8.0567e-12  & 8.0567e-12 \\
 & ite. & 6 & 6  & 6 \\
 & $\vc^*$ & \multicolumn{3}{l}{$\tran{(2.0000,3.2926, 3.4471, 4.9014, 6.5529)}$}\\
(b) & errs & 3.8846e-03  & 3.8846e-03  & 3.8846e-03 \\
 &  & 1.2491e-06  & 1.2491e-06  & 1.2491e-06 \\
 &  & 7.9990e-14  & 7.9990e-14  & 7.9990e-14 \\
 & ite. & 9  & 9  & 9 \\
 & $\vc^*$ & \multicolumn{3}{l}{$\tran{(2, 3, 4, 5, 6)}$}\\
(c) & errs & 5.0389e-05  & 5.0389e-05  & 5.0389e-05 \\
 &  & 3.5838e-10  & 3.5838e-10  & 3.5838e-10 \\
 &  & 5.0337e-15  & 5.0337e-15  & 5.0337e-15 \\
 & ite. & 13 & 13  & 13 \\
 & $\vc^*$ & \multicolumn{3}{l}{$\tran{(2.0000,3.2926, 3.4471, 4.9014, 6.5529)}$}\\
(d) & errs & 1.3806e-04  & 1.3806e-04  & 1.3806e-04 \\
 &  & 6.4118e-09  & 6.4118e-09  & 6.4118e-09 \\
 &  & 9.8493e-15  & 9.8493e-15  & 9.8493e-15 \\
 & ite. & 75 & 75  & 75 \\
 & $\vc^*$ & \multicolumn{3}{l}{$\tran{(2.0000,3.2926, 3.4471, 4.9014, 6.5529)}$}\\
(e) & errs & 1.2028e-04  & 7.7624e-03  & 7.7624e-03 \\
 &  & 3.8608e-10  & 6.1243e-06  & 6.1243e-06 \\
 &  & 5.3680e-15  & 2.1634e-12  & 2.1634e-12 \\
 & ite. & 13 & 5  & 5 \\
 & $\vc^*$ & \multicolumn{3}{l}{$\tran{(2, 3, 4, 5, 6)}$}\\
\hline
\end{tabular}
\end{center}
\caption{Numerical results for Example 1. } \label{tab:example1}
\end{table}

EXAMPLE 2. This is a Toeplitz-plus-Hankel inverse eigenvalue problem
(see \cite{Diele2004} for more details on this inverse problem) with
distinct eigenvalues. The basis matrices $\{A_i\}_{i = 1}^{7}$ are
given as follows:
\begin{align*}
A_1 & = \left[\begin{array}{rrrrrrr}
-1 & 0 & 0 & 0 & 0  & 0 & 0\\
0 & 1 & 0 & 0 & 0  & 0 & 0\\
0 & 0 & 1 & 0 & 0  & 0 & 0\\
0 & 0 & 0 & 1 & 0  & 0 & 0\\
0 & 0 & 0 & 0 & 1  & 0 & 0\\
0 & 0 & 0 & 0 & 0  & 1 & 0 \\
0 & 0 & 0 & 0 & 0  & 0 & 1
\end{array}\right], \ \ \
A_2 = \left[\begin{array}{rrrrrrr}
0 & -1 & 0 & 0 & 0  & 0 & 0\\
-1 & 0 & 1 & 0 & 0  & 0 & 0\\
0 & 1 & 0 & 1 & 0   & 0 & 0\\
0 & 0 & 1 & 0 & 1  & 0 & 0\\
0 & 0 & 0 & 1 & 0  & 1 & 0\\
0 & 0 & 0 & 0 & 1  & 0 & 1 \\
0 & 0 & 0 & 0 & 0  & 1 & 0
\end{array}\right],
\\
A_3 & = \left[\begin{array}{rrrrrrr}
0 & 0 & -1 & 0 & 0  & 0 & 0\\
0 & -2 & 0 & 1 & 0  & 0 & 0\\
-1 & 0 & 0 & 0 & 1  & 0 & 0\\
0 & 1 & 0 & 0 & 0  & 1 & 0\\
0 & 0 & 1 & 0 & 0  & 0 & 1\\
0 & 0 & 0 & 1 & 0  & 0 & 0 \\
0 & 0 & 0 & 0 & 1  & 0 & 0
\end{array}\right], \ \ \
A_4  = \left[\begin{array}{rrrrrrr}
0 & 0 & 0 & -1 & 0  & 0 & 0\\
0 & 0 & -2 & 0 & 1  & 0 & 0\\
0 & -2 & 0 & 0 & 0  & 1 & 0\\
-1 & 0 & 0 & 0 & 0  & 0 & 1\\
0 & 1 & 0 & 0 & 0  & 0 & 0 \\
0 & 0 & 1 & 0 & 0  & 0 & 0 \\
0 & 0 & 0 & 1 & 0  & 0 & 0
\end{array}\right],
\\
A_5 & = \left[\begin{array}{rrrrrrr}
0 & 0 & 0 & 0 & -1  & 0 & 0\\
0 & 0 & 0 & -2 & 0  & 1 & 0\\
0 & 0 & -2 & 0 & 0  & 0 & 1\\
0 & -2 & 0 & 0 & 0  & 0 & 0\\
-1 & 0 & 0 & 0 & 0  & 0 & 0 \\
0 & 1 & 0 & 0 & 0  & 0 & 0 \\
0 & 0 & 1 & 0 & 0  & 0 & 0
\end{array}\right], \ \ \
A_6 = \left[\begin{array}{rrrrrrr}
0 & 0 & 0 & 0 & 0  & -1 & 0\\
0 & 0 & 0 & 0 & -2  & 0 & 1\\
0 & 0 & 0 & -2 & 0  & 0 & 0\\
0 & 0 & -2 & 0 & 0  & 0 & 0\\
0 & -2 & 0 & 0 & 0  & 0 & 0 \\
-1 & 0 & 0 & 0 & 0  & 0 & 0 \\
0 & 1 & 0 & 0 & 0  & 0 & 0
\end{array}\right],
\\
A_7 & = \left[\begin{array}{rrrrrrr}
0 & 0 & 0 & 0 & 0  & 0 & -1\\
0 & 0 & 0 & 0 & 0  & -2 & 0\\
0 & 0 & 0 & 0 & -2  & 0 & 0\\
0 & 0 & 0 & -2 & 0  & 0 & 0\\
0 & 0 & -2 & 0 & 0  & 0 & 0 \\
0 & -2 & 0 & 0 & 0  & 0 & 0 \\
-1 & 0 & 0 & 0 & 0  & 0 & 0
\end{array}\right].
\end{align*}
The given real eigenvalues are
$$
\vec{\lambda}^* = \tran{(-35.4513, -13.6805, -9.5675, -8.5489,
8.7666,11.8220, 20.6596)},
$$
and $\vc^* = \tran{(2, 3, 4, 5, 6, 7, 8)}$ is a solution. In Table
\ref{tab:example2}, we report our numerical results for various
starting points:
\begin{align*}
(\text{a}) \ \ \vc^0 & = \tran{(1,2,3,4,5,6,7)};\ \ \ (\text{b}) \ \
\vc^0  = \tran{(1,3,5,7,9,11,13)};\\
(\text{c}) \ \ \vc^0 & = \tran{(11,13,15,17,19,21,23)};\ \
(\text{d}) \ \ \vc^0  = \tran{(50,52,56,58,62,65,68)}; \\
(\text{e}) \ \ \vc^0 & = \tran{(101,102,103,104,106,108,110)}.
\end{align*}

\begin{table}[t]
\begin{center}
\begin{tabular}{l@{\hspace{1.2cm}}l@{\hspace{1.2cm}}l@{\hspace{1.5cm}}l@{\hspace{1.5cm}}l@{\hspace{0.3cm}}}
\hline
$\vc^0$ &  & $\beta = 1.5 $  & $\beta = 1.8 $  & $\beta = 2.0 $ \\
\hline
(a) & errs & 8.2730e-03  & 8.2730e-03  & 8.2730e-03\\
 &  & 2.2866e-05  & 2.2866e-05  & 2.2866e-05 \\
 &  & 5.5845e-11  & 5.5845e-11  & 5.5845e-11 \\
 & ite. & 6 & 6  & 6 \\
 & $\vc^*$ & \multicolumn{3}{l}{$\tran{(2, 3, 4, 5, 6, 7, 8)}$}\\
(b) & errs & 8.2129e-03  & 4.0190e-03  & 1.7235e-04 \\
 &  & 2.0027e-05  & 1.5988e-06  & 3.7839e-09 \\
 &  & 5.0828e-11  & 2.8276e-13  & 1.9386e-14 \\
 & ite. & 7  & 7  & 7 \\
 & $\vc^*$ & \multicolumn{3}{l}{$\tran{(2.8703, 2.0639, 4.9434, 4.1053, 6.7530, 6.3287, 8.4794)}$}\\
(c) & errs & 3.4941e-03  & 3.4941e-03  & 3.4941e-03 \\
 &  & 2.5608e-06  & 2.5608e-06  & 2.5608e-06 \\
 &  & 6.7550e-13  & 6.7550e-13  & 6.7550e-13 \\
 & ite. & 11 & 11  & 11 \\
 & $\vc^*$ & \multicolumn{3}{l}{$\tran{(2.8703, 2.0639, 4.9434, 4.1053, 6.7530, 6.3287, 8.4794)}$}\\
(d) & errs & 2.9914e-04  & 2.9914e-04  & 2.9914e-04 \\
 &  & 8.2616e-09  & 8.2616e-09  & 8.2616e-09 \\
 &  & 1.8946e-14  & 1.8946e-14  & 1.8946e-14 \\
 & ite. & 13 & 13  & 13 \\
 & $\vc^*$ & \multicolumn{3}{l}{$\tran{(1.1787, -0.0035, 2.0401, 1.6976, 5.3794, 6.2068, 8.5273)}$}\\
(e) & errs & 2.1008e-04  & 2.1008e-04  & 2.1008e-04 \\
 &  & 1.5703e-08  & 1.5703e-08  & 1.5703e-08 \\
 &  & 2.2331e-14  & 2.2331e-14  & 2.2331e-14 \\
 & ite. & 13 & 13  & 13 \\
 & $\vc^*$ & \multicolumn{3}{l}{$\tran{(2.8703, 2.0639, 4.9434, 4.1053, 6.7530, 6.3287, 8.4794)}$}\\
\hline
\end{tabular}
\end{center}
\caption{Numerical results for Example 2. } \label{tab:example2}
\end{table}

We observe from Tables 1 and 2 that our algorithm is convergent for
different starting points. We also see that our algorithm converges
to a solution of the IEP, which is not necessarily equal to the
original one.


\end{document}